\documentclass{psp}
\usepackage{amsmath}
\usepackage{amsfonts}
\usepackage{amssymb}
\usepackage{graphicx}

\newtheorem{lemma}{Lemma}[section]
\newtheorem{theorem}{Theorem}[section]
\newtheorem{corollary}{Corollary}[section]

\newnumbered{remark}{Remark}[section]
\newnumbered{example}{Example}

\newcommand{\Z}{\mathbb{Z}}
\newcommand{\R}{\mathbb{R}}
\newcommand{\C}{\mathbb{C}}
\newcommand{\N}{\mathbb{N}}

\newcommand{\F}{\mathcal{F}}
\newcommand{\J}{\mathcal{J}}
\newcommand{\K}{\mathcal{K}}
\newcommand{\M}{\mathcal{M}}
\newcommand{\BigOh}{\mathcal{O}}
\newcommand{\eps}{\varepsilon}

\renewcommand{\phi}{\varphi}
\renewcommand{\theta}{\vartheta}

\DeclareMathOperator*{\Res}{\mathrm{Res}}
\DeclareMathOperator{\Arg}{\mathrm{arg}}

\numberwithin{equation}{section}

\begin{document}
\title[Poincar\'e functions and Julia sets]
{Complex asymptotics of Poincar\'e functions and properties of Julia sets}

\author[G. Derfel, P. J. Grabner and F. Vogl]
{GREGORY DERFEL\\
Department of Mathematics and Computer Science,\\
Ben Gurion University of the Negev,
Beer Sheva 84105, Israel\\
e-mail\textup{: \texttt{derfel@math.bgu.ac.il}}
\nextauthor{} PETER J. GRABNER\thanks{This author is supported by the
Austrian Science Foundation FWF, project S9605, part of the Austrian
National Research Network ``Analytic Combinatorics and Probabilistic
Number Theory''.}\\
Institut f\"ur Analysis und Computational Number Theory
(Math  A),\\
Technische Universit\"at Graz,
Steyrergasse 30,
8010 Graz, Austria\\
e-mail\textup{: \texttt{peter.grabner@tugraz.at}}\\
\and{} FRITZ VOGL\\
Institut f\"ur Analysis und Scientific Computing,
Technische Universit\"at Wien,\\
Wiedner Hauptstra\ss{}e 8--10,
1040 Wien, Austria\\
e-mail\textup{: \texttt{F.Vogl@gmx.at}}}

\maketitle

\begin{center}\emph{ Dedicated to Robert F. Tichy on the occasion of his
    \textup{50}\textsuperscript{th} birthday.}
\end{center}

\begin{abstract}
  The asymptotic behaviour of the solutions of Poincar\'e's functional equation
  $f(\lambda z)=p(f(z))$ ($\lambda>1$) for $p$ a real polynomial of degree
  $\geq2$ is studied in angular regions $W$ of the complex plain. It is known
  \cite{Derfel_Grabner_Vogl2007:asymptotics_poincare_functions,
    Derfel_Grabner_Vogl2008:zeta_function_laplacian} that $f(z)\sim\exp(z^\rho
  F(\log_\lambda z))$, if $f(z)\to\infty$ for $z\to\infty$ and $z\in W$, where
  $F$ denotes a periodic function of period $1$ and $\rho=\log_\lambda\deg(p)$.
  In the present paper we refine this result and derive a full asymptotic
  expansion. The constancy of the periodic function $F$ is characterised in
  terms of geometric properties of the Julia set of $p$.  For real Julia sets
  we give inequalities for multipliers of Pommerenke-Levin-Yoccoz type. The
  distribution of zeros of $f$ is related to the harmonic measure on the Julia
  set of $p$.
\end{abstract}

\section{Introduction}\label{sec:introduction}
\subsection{Historical remarks}\label{sec:historical-remarks}

In his seminal papers
\cite{Poincare1886:une_classe_etendue,Poincare1890:une_classe_nouvelle}
H.~Poincar\'e has studied the equation
\begin{equation}\label{Eq 1}
f(\lambda z)= R(f(z)),\quad z \in \C,
\end{equation}
where $R(z)$ is a rational function and $\lambda\in\C$.  He proved
that, if $R(0)=0$, $R'(0)=\lambda$, and $|\lambda|>1$, then there exists a
meromorphic or entire solution of (\ref{Eq 1}).  After Poincar\'e, (\ref{Eq 1})
is called {\em the Poincar\'e equation} and solutions of (\ref{Eq 1}) are
called {\em the Poincar\'e functions }. The next important step was made by
G.~Valiron
\cite{Valiron1923:lectures_on_general,Valiron1954:fonctions_analytiques}, who
investigated the case, where $R(z)=p(z)$ is a polynomial, i.e.
\begin{equation}\label{eq:poincare}
f(\lambda z)=p(f(z)),\quad z \in \C,
\end{equation}
and obtained conditions for the existence of an entire solution $f(z)$.
Furthermore, he derived the following asymptotic formula for
$M(r)=\max_{|z|\leq r}|f(z)|$:
\begin{equation}\label{Eq 3}
\log M(r)\sim r^{\rho}F\left(\frac{\log r}{\log |\lambda|}\right),
\quad r\rightarrow \infty.
\end{equation}
Here $F(z)$ is a $1$-periodic function bounded between two positive constants,
$\rho=\frac{\log d}{\log |\lambda|}$ and $d=\deg p(z)$. 

Different aspects of the Poincar\'e functions have been studied in the papers
\cite{Derfel_Grabner_Vogl2007:asymptotics_poincare_functions,
Derfel_Grabner_Vogl2008:zeta_function_laplacian,
Eremenko_Levin1989:periodic_points_polynomials,
Eremenko_Sodin1990:iterations_rational_functions,
Ishizaki_Yanagihara2005:borel_and_julia,
Romanenko_Sharkovsky2000:long_time_properties}. In particular in
\cite{Derfel_Grabner_Vogl2007:asymptotics_poincare_functions}, in addition to
(\ref{Eq 3}), asymptotics of entire solutions $f(z)$ on various rays
$re^{i \vartheta}$ of the complex plane have been found.

It turns out that this asymptotic behaviour heavily depends on the arithmetic
nature of $\lambda$. For instance, if $\Arg\lambda=2\pi\beta$, and $\beta$ is
irrational, then $f(z)$ is unbounded along any ray $\Arg z=\theta$
(cf.~\cite{Derfel_Grabner_Vogl2007:asymptotics_poincare_functions}).
\subsection{Assumptions}\label{sec:assumptions}
In the present paper we concentrate on the simplest, but maybe most important
case for applications, namely, when $\lambda$ is real and $p(z)$ is a real
polynomial (i.~e. all coefficients of $p(z)$ are real).

It is known from \cite{Valiron1954:fonctions_analytiques} and
\cite{Derfel_Grabner_Vogl2007:asymptotics_poincare_functions} that, if $f(z)$
is an entire solution of \eqref{eq:poincare}, then the only admissible values
for $f_0=f(0)$ are the fixed points of $p(z)$ (i.~e. $p(f_0)=f_0$). Moreover,
entire solutions exist, if and only if there exists an $n_0\in\N$ such that
\begin{equation*}
\lambda^{n_0}=p'(f_0).
\end{equation*}
It was proved in
\cite[Propositions~2.1--2.3]{Derfel_Grabner_Vogl2007:asymptotics_poincare_functions}
that the general case may be reduced to the simplest case
\begin{equation*}
f(0)=p(0)=0\text{ and }p'(0)=\lambda>1
\end{equation*}
by a change of variables. In the same vein, we can assume without loss of
generality that $f'(0)=1$ and the polynomial $p$ is monic (i.~e. the leading
coefficient is $1$)
\begin{equation}\label{eq:poly}
  p(z)=z^d+p_{d-1}z^{d-1}+\cdots+p_1z.
\end{equation}

\subsection{Poincar\'e and Schr\"oder equations}
\label{sec:poinc-schr-equat}
The functional equation \eqref{eq:poincare} with the additional (natural)
conditions $f(0)=0$ and $f'(0)=1$ is closely related to Schr\"oder's functional
equation (cf.~\cite{Schroeder1871:uber_iterierte_funktionen})
\begin{equation}\label{eq:schroeder}
g(p(z))=\lambda g(z),\quad g(0)=0\text{ and }g'(0)=1
\end{equation}
which was used by G.~Koenigs~\cite{Koenigs1884:recherches_sur_integrales,
  Koenigs1885:nouvelles_recherches_sur} to study the behaviour of $p$ under
iteration around the repelling fixed point $z=0$. By definition, $g$ is the
local inverse of $f$ around $z=0$. Both functions together provide a
linearisation of $p$ around its repelling fixed point $z=0$
\begin{equation*}
g(p(f(z)))=\lambda z\text{ and }g(p^{
(n)}(f(z)))=\lambda^n z,
\end{equation*}
where $p^{(n)}(z)$ denotes the $n$-th iterate of $p$ given by $p^{(0)}(z)=z$
and $p^{(n+1)}(z)=p(p^{(n)}(z))$.

We note here that \eqref{Eq 1} and \eqref{eq:poincare} are also called
Schr\"oder equation by some authors. For instance, the value distribution of
solutions of the Poincar\'e (alias Schr\"oder) equation \eqref{Eq 1} has been
investigated in \cite{Ishizaki_Yanagihara2005:borel_and_julia}.

\subsection{Branching processes and diffusion on fractals}
\label{sec:branch-proc-diff}
Iterative functional equations occur in the context of branching processes
(cf.~\cite{Harris1963:theory_branching_processes}). Here a probability
generating function
\begin{equation*}
q(z)=\sum_{n=0}^\infty p_nz^n
\end{equation*}
encodes the offspring distribution, where with $p_n\geq0$ is the probability
that an individual has $n$ offspring in the next generation (note that
$q(1)=1$). The growth rate $\lambda=q'(1)$ decides whether the population is
increasing ($\lambda>1$) or dying out $\lambda\leq1$. In the first case the
branching process is called \emph{super-critical}. The probability generating
function $q^{(n)}(z)$ ($n$-th iterate of $q$) encodes the distribution of the
size $X_n$ of the $n$-th generation under the offspring distribution $q$. In
the case of a super-critical branching process it is known that the random
variables $\lambda^{-n}X_n$ tend to a limiting random variable $X_\infty$. The
moment generating function of this random variable
\begin{equation*}
f(z)=\mathbb{E}e^{-zX_\infty}
\end{equation*}
satisfies the functional equation
(cf.~\cite{Harris1963:theory_branching_processes})
\begin{equation*}
f(\lambda z)=q(f(z)),
\end{equation*}
which is \eqref{eq:poincare}, if $q$ is a polynomial. Furthermore, this
equation can be transformed into \eqref{eq:poincare}, if $q$ is conjugate to a
polynomial by a M\"obius transformation, especially $q(z)=\frac1{p(1/z)}$,
where $p$ is a polynomial.

Branching processes have been used in \cite{Barlow1998:diffusions_on_fractals,
  Barlow_Perkins1988:brownian_motion_sierpinski,
  Lindstroem1990:brownian_motion_nested} to model time for the Brownian motion
on certain types of self-similar structures such as the Sierpi\'nski gasket. In
this context the zeros of the solution of \eqref{eq:poincare} are the
eigenvalues of the infinitesimal generator of the diffusion (``Laplacian''), if
the generating function of the offspring distribution is conjugate to a
polynomial (cf.~\cite{Derfel_Grabner_Vogl2008:zeta_function_laplacian,
  Grabner1997:functional_iterations_stopping,
  Malozemov_Teplyaev2003:self_similarity_operators,
  Teplyaev2004:spectral_zeta_function,
  Teplyaev2007:spectral_zeta_functions}).
In this case the zeros of $f$ have to be real, since they are eigenvalues of a
self-adjoint operator. This motivates the investigation of real Julia sets in
Section~\ref{sec:real-julia-set}.

\subsection{Contents}\label{sec:contents}
The paper is organised as follows.

In Section~\ref{sec:asympt-infin-fatou} we study the asymptotic behaviour of
$f(z)$ in those sectors $W$ of the complex plane, where
\begin{equation}\label{eq:infty}
f(z)\to\infty \text{ for } z\to\infty,\quad z\in W.
\end{equation}
It was proved in
\cite{Derfel_Grabner_Vogl2007:asymptotics_poincare_functions,
Derfel_Grabner_Vogl2008:zeta_function_laplacian} that \eqref{eq:infty} implies
\begin{equation*}
f(z)\sim\exp\left(z^\rho F\left(\frac{\log z}{\log\lambda}\right)\right)
\text{ for }z\to\infty,\quad z\in W,
\end{equation*}
where $F(z)$ is a periodic function of period $1$. In
Section~\ref{sec:asympt-infin-fatou} we will refine this result to a full
asymptotic expansion of $f(z)$, which takes the form
\begin{equation}\label{eq:f-asymp-1}
  f(z)=\exp\left(z^\rho 
    F\left(\log_\lambda z\right)\right)+
  \sum_{n=0}^\infty c_n\exp\left(-nz^\rho 
    F\left(\log_\lambda z\right)\right),
\end{equation}
where $F$ is a periodic function of period $1$ holomorphic in some strip
depending on $W$ and $\rho=\log_\lambda d$. The proof is based on an
application of the B\"ottcher function at $\infty$ of $p(z)$.

We note here that E.~Romanenko and A.~Sharkovsky
\cite{Romanenko_Sharkovsky2000:long_time_properties} have studied equation
\eqref{eq:poincare} on $\R$ (rather than $\C$) and obtained a full asymptotic
expansion of this type by Sharkovsky's method of ``first integrals'' or
``invariant curves''.

Further analysis of the periodic function $F$ occurring in \eqref{eq:f-asymp-1}
is presented in Section~\ref{sec:furth-analys-peri}, where the Fourier
coefficients of $F$ are related to the B\"ottcher function at $\infty$ of
$p(z)$ and the harmonic measure on the Julia set of $p$.

In Section~\ref{sec:asympt-finite-fatou} the asymptotic behaviour of $f(z)$ is
studied in sectors that are related to basins of attraction of finite
attracting fixed points.

In Section~\ref{sec:zeros-poinc-funct} we relate geometric properties of the
Julia set to the location of the zeros of $f$.

Section~\ref{sec:real-julia-set} is devoted to the special case of real Julia
sets $\J(p)$. Here we prove, in particular, the following inequalities of
Pommerenke-Levin-Yoccoz type for multipliers of fixed points $\xi$:
\begin{equation}\label{eq:pommerenke}
p(\xi)=\xi\Rightarrow
\begin{cases}
  |p'(\xi)|\geq d&\text{ for }\min\J(p)<\xi<\max\J(p)\\
  |p'(\xi)|\geq d^2&\text{ for }\xi=\min\J(p)\text{ or }\xi=\max\J(p).
\end{cases}
\end{equation}
Furthermore, equality can hold only, if $p$ is linearly conjugate to a
Chebyshev polynomial of the first kind.

In Section~\ref{sec:zeta-funct-poinc} we continue the study of Dirichlet
generating functions of zeros of Poincar\'e functions that we started in
\cite{Derfel_Grabner_Vogl2008:zeta_function_laplacian} in the context of
spectral zeta functions on certain fractals. We relate the poles and
residues of the zeta function of $f$ to the Mellin transform of the harmonic
measure $\mu$ on the Julia set of $p$. Furthermore, we show a connection
between the zero counting function of $f$ and the harmonic measure $\mu$ of
circles around the origin.

\section{Relation of complex asymptotics and the Fatou set}
\label{sec:relat-compl-asympt}
Throughout the rest of the paper we will use the following notations and
assumptions. Let $p$ be a real polynomial of degree $d$ as in \eqref{eq:poly}.
We always assume that $p(0)=0$ and $p'(0)=a_1=\lambda$ with $|\lambda|>1$. We
refer to
\cite{Beardon1991:iteration_rational_functions,Milnor2006:dynamics_complex} as
general references for complex dynamics.

We denote the Riemann sphere by $\C_\infty$ and consider $p$ as a map on
$\C_\infty$. We recall that the Fatou set $\F(p)$ is the set of all
$z\in\C_\infty$ which have an open neighbourhood $U$ such that the sequence
$(p^{(n)})_{n\in\N}$ is equicontinuous on $U$ in the chordal metric on
$\C_\infty$. By definition $\F(p)$ is open. We will especially need the
component of $\infty$ of $\F(p)$ given by
\begin{equation}\label{eq:Fatou-infty}
\F_\infty(p)=\left\{z\in\C\mid \lim_{n\to\infty}p^{(n)}(z)=\infty\right\},
\end{equation}
as well as the basins of attraction of a finite attracting fixed point $w_0$
($p(w_0)=w_0$, $|p'(w_0)|<1$)
\begin{equation}\label{eq:Fatou-w0}
\F_{w_0}(p)=\left\{z\in\C\mid \lim_{n\to\infty}p^{(n)}(z)=w_0\right\}.
\end{equation}
The complement of the Fatou set is the Julia set
$\J(p)=\C_\infty\setminus\F(p)$.

The filled Julia set is given by
\begin{equation}\label{eq:filled-Julia}
\K(p)=\left\{z\in\C\mid (p^{(n)}(z))_{n\in\N}\text{ is bounded}\right\}=
\C\setminus\F_\infty(p).
\end{equation}
Furthermore, it is known that (cf.~\cite{Falconer2003:fractal_geometry})
\begin{equation}\label{eq:boundary}
\partial\K(p)=\partial\F_\infty(p)=\J(p).
\end{equation}
In the case of polynomials this can be used as an equivalent definition of the
Julia set.

We will also use the notations
\begin{equation}\label{eq:W_alpha_beta}
W_{\alpha,\beta}=\left\{z\in\C\setminus\{0\}\mid \alpha<\arg z<\beta\right\}
\end{equation}
and
\begin{equation*}
B(z,r)=\left\{w\in\C\mid |z-w|<r\right\}.
\end{equation*}

\subsection{Asymptotics in the infinite Fatou component}
\label{sec:asympt-infin-fatou}

In \cite{Derfel_Grabner_Vogl2007:asymptotics_poincare_functions,
Derfel_Grabner_Vogl2008:zeta_function_laplacian} the asymptotics
of the solution of the Poincar\'e equation \eqref{eq:poincare} was
given. We want to present a different approach here, which gives a full
asymptotic expansion. 
\begin{theorem}\label{thm:poincare-asymp}
Let $f$ be the entire solution of the Poincar\'e equation \eqref{eq:poincare}
for a real polynomial $p$ with $\lambda=p'(0)>1$. Assume further that the Fatou
component of $\infty$, $\F_\infty(p)$ contains an angular region
$W_{\alpha,\beta}$. 
\begin{description}
\item[A] Then the following asymptotic expansion for $f$  is valid for all
  $z\in W_{\alpha,\beta}$ large enough
  \begin{equation}\label{eq:f-asymp}
    f(z)=\exp\left(z^\rho 
      F\left(\log_\lambda z\right)\right)+
    \sum_{n=0}^\infty c_n\exp\left(-nz^\rho 
      F\left(\log_\lambda z\right)\right),
  \end{equation}
where $F$ is a
 periodic function  of period $1$ holomorphic
  in the strip
\begin{equation*}
  \left\{z\in\C\mid \frac{\alpha}{\log\lambda}<\Im z<\frac{\beta}{\log\lambda}
    \right\}
\end{equation*}
and $\rho=\log_\lambda d$. Furthermore,
  \begin{equation}\label{eq:Re>0}
    \forall z\in W_{\alpha,\beta}:\Re z^\rho F(\log_\lambda z)>0
  \end{equation}
  holds.
\item[B] Let $g$ denote the B\"ottcher function associated with $p$,
  \emph{i.~e.}
\begin{equation}\label{eq:boettcher}
(g(z))^d=g(p(z))
\end{equation}
in some neighbourhood of $\infty$.  Its inverse function is given by the
Laurent series around $\infty$
  \begin{equation}\label{eq:Boettcher-inverse-cn}
    g^{(-1)}\left(w\right)=w+\sum_{n=0}^\infty\frac{c_n}{w^n}.
  \end{equation}
Then we have
\begin{equation*}
f(z)=g^{(-1)}\left(\exp\left(z^\rho 
      F\left(\log_\lambda z\right)\right)\right)
\end{equation*}
and $c_n$ can be determined from the coefficients of $p$.
\end{description}
\end{theorem}
\begin{proof}
  We recall that $p$ has a super-attracting fixed point of order $d=\deg p$ at
  infinity. We consider the B\"ottcher function $g$ associated with this fixed
  point (cf.~\cite{Beardon1991:iteration_rational_functions,
Blanchard1984:complex_analytic_dynamics,
    Boettcher1905:beitraege_zur_theorie,
    Kuczma_Choczewski_Ger1990:iterative_functional_equations}), which satisfies
  the functional equation \eqref{eq:boettcher}
in some neighbourhood of infinity. The B\"ottcher function has a
Laurent expansion around infinity given by
\begin{equation}\label{eq:Boettcher-Laurent}
g(z)=z+\sum_{n=0}^\infty\frac{b_n}{z^n},
\end{equation}
which converges for $|z|>R$ for some $R>0$. The coefficients $(b_n)_{n\in\N_0}$
can be determined uniquely from the coefficients of the polynomial $p$.

Using the B\"ottcher function we can rewrite the Poincar\'e equation assuming
that $|f(z)|>R$
\begin{equation}\label{eq:Poincare-Boettcher}
(g(f(z)))^d=g(p(f(z)))=g(f(\lambda z)).
\end{equation}
From this we derive that $h(z)=g(f(z))$ satisfies the much simpler functional
equation
\begin{equation*}
(h(z))^d=h(\lambda z),
\end{equation*}
which only holds for those values $z$ for which $|f(z)|>R$. This equation has
solutions
\begin{equation}\label{eq:h(z)}
h(z)=\exp\left(z^\rho F\left(\log_\lambda z\right)\right)
\end{equation}
with $\rho=\log_\lambda d$ and $F$ a periodic function of period
$1$ holomorphic in some strip parallel to the real axis. Since $|h(z)|>1$ for
all $z$ with $|f(z)|>R$ by the properties of the function $g$, we have
\eqref{eq:Re>0}.

By \eqref{eq:Boettcher-Laurent} $g$ is invertible in some neighbourhood of
$\infty$ and we can write \eqref{eq:Boettcher-inverse-cn}
where the coefficients $c_n$ depend only on the coefficients of the
polynomial $p$. This function satisfies the functional equation
\begin{equation}\label{eq:Boettcher-inverse}
g^{(-1)}(w^d)=p(g^{(-1)}(w))
\end{equation}
for $w$ in some neighbourhood of $\infty$. Inserting \eqref{eq:h(z)} into
\eqref{eq:Boettcher-inverse-cn} yields \eqref{eq:f-asymp}
giving an exact and asymptotic expression for $f(z)$.
\end{proof}
\begin{remark}\label{rem8}
  E.~Romanenko and A.~Sharkovsky have studied equation
  \eqref{eq:poincare} on $\R$ (rather than on $\C$)
  in \cite{Romanenko_Sharkovsky2000:long_time_properties}.
  Applying Sharkovsky's method of ``first integrals'' (``invariant graphs'')
  they obtained a full asymptotic formula of type \eqref{eq:f-asymp} for all
  solutions $f(x)$, such that $f(x)\to\infty$ for $x\to\infty$.
\end{remark}

\subsection{B\"ottcher functions, Green functions, and constancy 
of the periodic function $F$}\label{sec:bottch-funct-green}
We will make frequent use of the integral representation of the
B\"ottcher function
\begin{equation}\label{eq:Boettcher-int}
g(z)=\exp\left(\int_{\J(p)}\log(z-x)\,d\mu(x)\right),
\end{equation}
where $\mu$ denotes the harmonic measure on the Julia set $\J(p)$
(cf.~\cite{Bessis_Geronimo_Moussa1984:mellin_transforms_associated,
Brolin1965:invariant_sets_under,
Ransford1995:potential_theory_complex_plane}). This shows
that $g$ is holomorphic on any simply connected subset of
$\F_\infty(p)$. The measure $\mu$ can be given as the weak
limit of the measures
\begin{equation}\label{eq:mu_n}
\mu_n=\frac1{d^n}\sum_{p^{(n)}(x)=\xi}\delta_x,
\end{equation}
where $\xi$ can be chosen arbitrarily (not exceptional) and $\delta_x$ denotes
the unit point mass at $x$ (cf.~\cite{Brolin1965:invariant_sets_under,
  Ransford1995:potential_theory_complex_plane}).

The function $g(z)$ can be continued to any simply connected subset $U$ of
$\C_\infty\setminus\K(p)$ (this follows for instance from the integral
representation \eqref{eq:Boettcher-int}). Furthermore, it follows
from \cite[Lemma~9.5.5]{Beardon1991:iteration_rational_functions} and
\eqref{eq:boettcher} that
\begin{equation*}
  g(U)\subset\{z\in\C_\infty\mid |z|>1\}.
\end{equation*}
The function $\log|g(z)|$ is the Green function for the logarithmic potential
on $\F_\infty(p)$
(cf.~\cite[Section~9]{Beardon1991:iteration_rational_functions}). Combining
classical potential theory with polynomial iteration theory we get
\begin{equation}\label{eq:Julia-condition}
\lim_{\substack{z\to z_0\\ z\in \F_\infty(p)}}|g(z)|=1\Leftrightarrow
z_0\in\J(p),
\end{equation}
where the implication $\Leftarrow$ is
\cite[Lemma~9.5.5]{Beardon1991:iteration_rational_functions}. The opposite
implication is a general property of the Green function
(cf.~\cite[Chapter~III]{Garnett_Marshall2005:harmonic_measure}, and
\cite[Section~6.5]{Ransford1995:potential_theory_complex_plane}) combined with
the fact that $\partial\F_\infty(p)=\J(p)$ for polynomial $p$.

\begin{theorem}\label{thm:constant}
The periodic function $F$ occurring in the asymptotic expression
\eqref{eq:f-asymp} for $f$ is constant, if and only if the polynomial $p$ is
either linearly conjugate to $z^d$ or to the Chebyshev polynomial of the first
kind $T_d(z)$.
\end{theorem}
\begin{proof}
The periodic function $F$ is constant, if and only if the function
$h(z)=g(f(z))$ introduced above satisfies
\begin{equation}\label{eq:h-exact}
h(z)=\exp\left(Cz^\rho\right)
\end{equation}
for some constant $C\neq0$. This implies that for any
$w_0\in\J(p)\setminus\{0\}$ the function $g$ has an analytic continuation to
some open neighbourhood of $w_0$. Thus 
\eqref{eq:Julia-condition} can be replaced by
\begin{equation*}
|g(w_0)|=1 \Leftrightarrow w_0\in\J(p)
\end{equation*}
in our case. By \eqref{eq:h-exact} this is equivalent to $w_0=f(z_0)$ for
$Cz_0^\rho\in i\R$.  Since $Cz^\rho\in i\R$ describes an analytic curve (with a
possible cusp at $z=0$), the Julia set of $p$ is the image of this curve under
the entire function $f$, thus itself an analytic arc.

By \cite[Theorem~1]{Hamilton1995:length_julia_curves} $\J(p)$ can only be an
analytic arc, if the Julia set of $p$ is either a line segment or a circle.
The Julia set is a line segment, if and only if $p$ is linearly conjugate to
the Chebyshev polynomial $T_d$
(cf.~\cite[Theorem~1.4.1]{Beardon1991:iteration_rational_functions}); the Julia
set is a circle, if and only if $p$ is linearly conjugate to $z^d$
(cf.~\cite[Theorem~1.3.1]{Beardon1991:iteration_rational_functions}).
\end{proof}
\begin{remark} Suppose that the periodic function $F$ is constant. If $p$ is
  linearly conjugate to a monomial, then the B\"ottcher function $g$ and
  therefore its inverse are linear functions. In this case $\rho=1$. (We recall
  that we generally assume that $f'(0)=1$.)  If $p$ is linearly conjugate to a
  Chebyshev polynomial, $g^{(-1)}$ is linearly conjugate to the Joukowski
  function $z+\frac1z$.  In this case $\rho=1$, if $0$ is an inner point of the
  line segment $\J(p)$, and $\rho=\frac12$, if $0$ is an end point of the line
  segment $\J(p)$ (cf.~Sections~\ref{sec:negative-julia-set}
  and~\ref{sec:julia-set-has}). Furthermore, the asymptotic series
  \eqref{eq:f-asymp} is finite, if the periodic function $F$ is constant.
\end{remark}
\subsection{Further analysis of the periodic function}
\label{sec:furth-analys-peri}
In this section we relate the periodic function $F$ occurring in
\eqref{eq:f-asymp} to the local behaviour of the B\"ottcher function at the
fixed point $f(0)=0$. 

This will allow to express the Fourier coefficients of $F$ in terms of residues
of the Mellin transform (cf.~\cite{Doetsch1971:handbuch_der_laplace,
  Oberhettinger1974:tables_mellin_transforms}) of the harmonic measure $\mu$
given by \eqref{eq:mu_n}. This Mellin transform was introduced and studied in
\cite{Bessis_Geronimo_Moussa1984:mellin_transforms_associated}.  A similar
relation was also used in \cite{Grabner1997:functional_iterations_stopping} to
derive an asymptotic expression for $f$ in a special case.

We will use the relation
\begin{equation}\label{eq:G(w)}
G(w)=\log g(w)=\int_{\J(p)}\log(w-x)\,d\mu(x)
\end{equation}
between the (complex) ``Green function'' $G$ and the B\"ottcher function $g$.
Assume that the Fatou component $\F_\infty(p)$ contains an angular region
centred at the fixed point $0$. Furthermore, assume that $\lim_{w\to 0}g(w)=1$.
Then \eqref{eq:h(z)} holds in this angular region. This fact can be used to
analyse the local behaviour of $\log g(w)$ around $w=0$:
\begin{equation}\label{eq:logg}
\log g(w)=\left(f^{(-1)}(w)\right)^\rho
F\left(\log_\lambda f^{(-1)}(w)\right)=
w^\rho F\left(\log_\lambda w\right)+\BigOh(w^{\rho+1}).
\end{equation}
Thus the behaviour of the Green function $G$
at the point $0$ exhibits the same periodic function $F$ as the asymptotic
expansion of $\log f$ around $\infty$.

\begin{figure}[h]
  \centering
  \includegraphics[width=0.8\hsize]{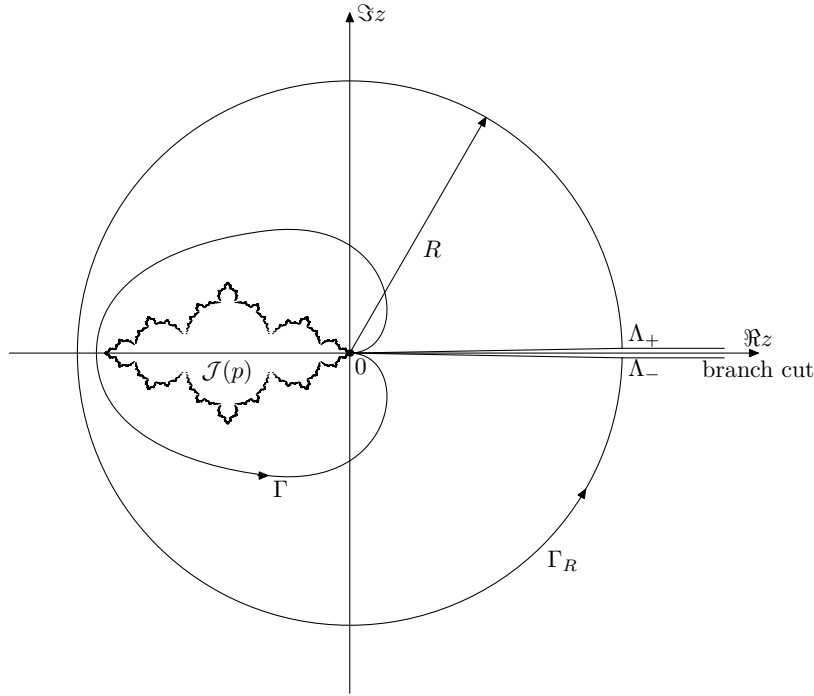}
  \caption{Paths of integration.}
  \label{fig:paths}
\end{figure}

We now relate the Green function $G(w)$ to the Mellin transform of $\mu$
\begin{equation}\label{eq:Mellin}
M_\mu(s)=\int_{\J(p)}(-x)^s\,d\mu(x),
\end{equation}
where the branch cut for the function $(-x)^s$ is chosen to connect $0$
with $\infty$ without any further intersection with $\J(p)$. Following the
computations in
\cite[Section~5]{Bessis_Geronimo_Moussa1984:mellin_transforms_associated} we
obtain
\begin{equation*}
M_\mu(s)=\frac1{2\pi i}\oint_\Gamma(-z)^s\,dG(z)=
\frac1{2\pi i}\oint_{\Gamma_R}(-z)^s\,dG(z).
\end{equation*}
For $\Re s<0$ we have for the circle of radius $R$
\begin{equation*}
\left|\frac1{2\pi i}\int_{|z|=R}(-z)^s\,dG(z)\right|\ll R^{\Re s},
\end{equation*}
which allows to let $R\to\infty$ in this case. This gives
\begin{multline*}
M_\mu(s)=\frac1{2\pi i}\left(\int_{\Lambda_+}(-z)^s\,dG(z)-
\int_{\Lambda_-}(-z)^s\,dG(z)\right)\\
=\frac{e^{-i\pi s}-e^{i\pi s}}{2\pi i}\int_0^\infty x^sG'(x)\,dx=
s\frac{\sin\pi s}\pi\int_0^\infty x^{s-1}G(x)\,dx,
\end{multline*}
which relates the Mellin transform of the measure $\mu$ to the Mellin transform
of the function $G(z)$
\begin{equation}\label{eq:Mellin-G}
\M G(s)=\int_0^\infty x^{s-1}G(x)\,dx=\frac\pi{s\sin\pi s}M_\mu(s)\text{ for }
-\rho<\Re s<0.
\end{equation}

The function $M_\mu(s)$ (and therefore $\M G(s)$ by \eqref{eq:Mellin-G}) has an
analytic continuation by the following observation
\begin{equation}\label{eq:continuation}
M_\mu(s)=\frac1d\sum_{k=1}^d\int_{\J(p)}(-p_k^{(-1)}(x))^s\,d\mu(x),
\end{equation}
where $p_k^{(-1)}$ ($k=1,\ldots,d$) denote the $d$ branches of the inverse
function of $p$; we choose the numbering so that $p_1^{(-1)}(0)=0$. The
summands for $k=2,\ldots,d$ are clearly entire functions in $s$, since the
integrand is bounded away from $0$ and $\infty$. For the summand with $k=1$ we
observe that
\begin{equation}\label{eq:approx}
p_1^{(-1)}(x)=\frac1\lambda x+\BigOh(x^2)\text{ for }x\to0.
\end{equation}
Inserting this into \eqref{eq:continuation} gives
\begin{multline*}
M_\mu(s)=\frac1d\lambda^{-s}\int_{\J(p)}(-x)^s\,d\mu(x)+
\frac1d\lambda^{-s}\int_{\J(p)}(-x)^s\BigOh(x)\,d\mu(x)\\+
\frac1d\sum_{k=2}^d\int_{\J(p)}(-p_k^{(-1)}(x))^s\,d\mu(x),
\end{multline*}
where the second term on the right-hand-side originates from inserting the
holomorphic function $\BigOh(x^2)$ from \eqref{eq:approx} into the
integrand, which gives a function holomorphic in a larger domain.
Thus we obtain
\begin{equation}\label{eq:Mellin-continuation}
M_\mu(s)=\frac1{d\lambda^s-1}H(s)
\end{equation}
for some function $H(s)$ holomorphic for $\Re s>-\rho-1$ 
($\rho=\log_\lambda d$). The numerator
$d\lambda^s-1$ has zeros at $s=-\rho+\frac{2k\pi i}{\log\lambda}$ ($k\in\Z$),
which give possible poles for the function $M_\mu(s)$.
\begin{remark}
  Using the full Taylor expansion of $p_1^{(-1)}(x)$ instead of the
  $\BigOh$-term in \eqref{eq:approx} would yield the existence of a meromorphic
  continuation of $M_\mu(s)$ to the whole complex plane.
\end{remark}

Taking \eqref{eq:Mellin-G} and \eqref{eq:Mellin-continuation} together gives
the analytic continuation of $\M G(s)$ to $-\rho-1<\Re s<0$. Then the Mellin
inversion formula (cf.~\cite{Doetsch1971:handbuch_der_laplace}) gives (for
$-\rho<c<0$)
\begin{multline}\label{eq:Mellin-inv-G}
G(x)=\frac1{2\pi i}\int\limits_{c-i\infty}^{c+i\infty}\M G(s)x^{-s}\,ds=
\frac1{2\pi i}\int\limits_{c-i\infty}^{c+i\infty}\frac\pi{s\sin\pi s}
\frac1{d\lambda^s-1}H(s)x^{-s}\,ds\\
=\frac1{2\pi i}\int\limits_{-\rho-\frac12-i\infty}^{-\rho-\frac12+i\infty}
\frac\pi{s\sin\pi s}
\frac1{d\lambda^s-1}H(s)x^{-s}\,ds+
\sum_{k\in\Z}\Res_{s=-\rho+\frac{2k\pi i}{\log\lambda}}
\M G(s)x^{-s}.
\end{multline}
The integral in the second line is $\BigOh(x^{\rho+\frac12})$, the sum of
residues can be evaluated further to give the Fourier expansion of the periodic
function $F$
\begin{equation}\label{eq:Fourier-F}
\sum_{k\in\Z}\Res_{s=-\rho+\frac{2k\pi i}{\log\lambda}}
\M G(s)x^{-s}=x^\rho\sum_{k\in\Z}f_k e^{2k\pi i\log_\lambda x}=
x^\rho F(\log_\lambda x).
\end{equation}
The Fourier coefficients $f_k$ are given by
\begin{multline}\label{eq:fk}
f_k=\Res_{s=-\rho-\frac{2k\pi i}{\log\lambda}}\M G(s)=
\frac\pi{\left(-\rho-\frac{2k\pi i}{\log\lambda}\right)
\sin\pi\left(-\rho-\frac{2k\pi i}{\log\lambda}\right)}
\Res_{s=-\rho-\frac{2k\pi i}{\log\lambda}}M_\mu(s)\\
=\frac\pi{\left(-\log d-2k\pi i\right)
\sin\pi\left(-\rho-\frac{2k\pi i}{\log\lambda}\right)}
H\left(-\rho-\frac{2k\pi i}{\log\lambda}\right).
\end{multline}

\subsection{Asymptotics in a finite Fatou component -- 
analysis of asymptotic values}
\label{sec:asympt-finite-fatou}
It is clear from the functional equation \eqref{eq:poincare} for $f$ that any
asymptotic value of $f$ has to be an attracting fixed point of the polynomial
$p$ (including $\infty$). Thus the analysis in
Section~\ref{sec:asympt-infin-fatou} can be interpreted as the behaviour of $f$
when approaching the asymptotic value $\infty$. In the present section we
extend this analysis to all asymptotic values.

First we study the case of a finite attracting, but not super-attracting fixed
point. Let $w_0$ be an attracting fixed point of $p$ and denote
$\eta=p'(w_0)\neq0$ ($|\eta|<1$).  Then there exists a solution $\Psi$ of the
Schr\"oder equation
\begin{equation}\label{eq:schroeder-Psi}
\eta\Psi(z)=\Psi(p(z)),\quad \Psi(w_0)=0,\text{ and }\Psi'(w_0)=1,
\end{equation}
which is holomorphic in $\F_{w_0}(p)$ (for instance, the sequence
$(\eta^{-n}(p^{(n)}(z)-w_0))_{n\in\N}$ converges to $\Psi$ on any compact
subset of $\F_{w_0}(p)$). Assume now that $\F_{w_0}(p)$ contains an angular
region $W_{\alpha,\beta}\cap B(0,r)$ for some $r>0$. Then by conformity of $f$
some angular region at the origin is mapped into $W_{\alpha,\beta}\cap B(0,r)$.
We consider the function
\begin{equation*}
j(z)=\Psi(f(z)),
\end{equation*}
which satisfies the functional equation
\begin{equation}\label{eq:Psif}
j(\lambda z)=\Psi(f(\lambda z))=\Psi(p(f(z)))=\eta\Psi(f(z))=\eta j(z).
\end{equation}
This equation has the solution
\begin{equation}\label{eq:jz}
j(z)=z^{\log_\lambda\eta}H(\log_\lambda z)
\end{equation}
with some periodic function of period $1$, holomorphic in some strip. This
periodic function can never be constant, since otherwise $j(z)$ would have an
analytic continuation to the slit complex plane. From this it would follow that
$f$ is bounded in the slit complex plane, a contradiction.

The function $\Psi$ has a holomorphic inverse around $0$
\begin{equation*}
\Psi^{(-1)}(z)=w_0+z+\sum_{n=2}^\infty \psi_nz^n
\end{equation*}
which allows us to write
\begin{equation}\label{eq:fw0}
f(z)=\Psi^{(-1)}\left(z^{\log_\lambda\eta}H(\log_\lambda z)\right)=
w_0+z^{\log_\lambda\eta}H(\log_\lambda z)+
\sum_{n=2}^\infty \psi_nz^{n\log_\lambda\eta}(H(\log_\lambda z))^n,
\end{equation}
which is valid in the angular region $W_{\alpha,\beta}$ for $z$ large enough.
This gives an exact and asymptotic expression for $f$ in an angular region.

In the case of a super-attracting fixed point $w_0$ we have $p'(w_0)=0$. Assume
that the first $k-1$ derivatives of $p$ vanish in $w_0$, but the $k$-th
derivative is non-zero. Then $p(z)=(z-w_0)^kP(z)$ with $P(w_0)=A\neq0$. We use
the solution $g$ of the corresponding B\"ottcher equation
\begin{equation}\label{eq:boettcher-w0}
g(p(z))=A(g(z))^k\quad g(w_0)=0,\quad g'(w_0)=1
\end{equation}
to linearise \eqref{eq:poincare}
\begin{equation*}
g(f(\lambda z))=g(p(f(z)))=A(g(f(z)))^k.
\end{equation*}
Thus the function $h(z)=g(f(z))$ satisfies
\begin{equation*}
h(\lambda z)=A(h(z))^k.
\end{equation*}
This equation has solutions
\begin{equation*}
h(z)=A^{-\frac1{k-1}}\exp\left(z^{\log_\lambda k}
L\left(\log_\lambda z\right)\right)
\end{equation*}
for a periodic function $L$ of period $1$ and a suitable choice of the
$(k-1)$-th root. Furthermore, by the fact that $\lim_{z\to\infty}h(z)=0$ we
have
\begin{equation*}
\Re\left(z^{\log_\lambda k} L\left(\log_\lambda z\right)\right)<0\text{ for }
f(z)\in\F_{w_0}(p).
\end{equation*}
using the local inverse of $g$ around $0$ we get
\begin{multline}\label{eq:f-asymp-w0}
f(z)=g^{(-1)}\left(A^{-\frac1{k-1}}\exp\left(z^{\log_\lambda k}
L\left(\log_\lambda z\right)\right)\right)\\
=w_0+A^{-\frac1{k-1}}\exp\left(z^{\log_\lambda k}
L\left(\log_\lambda z\right)\right)(1+o(1)).
\end{multline}

Summing up, we have proved
\begin{theorem}\label{thm10}
  Let $w_0$ be an attracting fixed point of $p$ such that the Fatou component
  $\F_{w_0}(p)$ contains an angular region $W_{\alpha,\beta}\cap B(0,r)$ for
  some $r>0$. Then the asymptotic behaviour of $f$ for $z\to\infty$ and $z\in
  W_{\alpha,\beta}$ is given by \eqref{eq:fw0}, if $\eta=p'(w_0)\neq0$, and by
  \eqref{eq:f-asymp-w0}, if $p(z)-w_0$ has a zero of order $k$ in $w_0$.
\end{theorem}
\begin{remark}
  The periodic function $H$ in \eqref{eq:fw0} cannot be constant, because
  otherwise $f(z)$ would be bounded. The periodic function $L$ in
  \eqref{eq:f-asymp-w0} can only be constant, if $p$ is linearly conjugate to
  $z^k$, by the same arguments as in the proof of Theorem~\ref{thm:constant}
  (the case of Chebyshev polynomials does not occur, because they only have
  repelling finite fixed points).
\end{remark}

As a consequence of Ahlfors' theorem on asymptotic values
(cf.~\cite{Goluzin1969:geometric_theory_functions}) and Valiron's theorem on
the growth of $f$
(cf.~\cite{Valiron1923:lectures_on_general,Valiron1954:fonctions_analytiques})
we get an upper bound for the number of attracting fixed points of a
polynomial.
\begin{theorem}
Let $p$ be a real polynomial of degree $d>1$ and let
\begin{equation*}
\gamma=\max\left\{|p'(z)|\mid p(z)=z\right\}.
\end{equation*}
Then the number of (finite) attracting fixed points of $p$ is
bounded by $2\log_\gamma d$, i.e.
\begin{equation}\label{eq:ahlfors}
\#\left\{z\in\C\mid p(z)=z\wedge |p'(z)|<1\right\}\leq 2\log_\gamma d.
\end{equation}
\end{theorem}

\section{Zeros of the Poincar\'e  function and Julia sets}
\label{sec:zeros-poinc-funct}
In this section we relate the distribution of zeros of the Poincar\'e function
in angular regions to geometric properties of the Julia set $\J(p)$ of the
polynomial $p$.
\begin{theorem}\label{thm8}
Let $p$ be a real polynomial with $p(0)=0$ and $p'(0)=\lambda>1$.
Then the following are equivalent
\begin{enumerate}
\item\label{thm8.1}
$\displaystyle{\forall r>0: W_{\alpha,\beta}\cap \J(p)\cap B(0,r)\neq\emptyset}$
\item\label{thm8.3} $W_{\alpha,\beta}$ contains a zero of $f$.
\item\label{thm8.2} $W_{\alpha,\beta}$ contains infinitely many zeros of $f$.
\end{enumerate}
\end{theorem}
\begin{proof}
We first remark that \ref{thm8.3} and \ref{thm8.2} are trivially
equivalent, since $f(z_0)=0$ implies that $f(\lambda^n z_0)=0$.

For the proof of ``\ref{thm8.1}$\Rightarrow$~\ref{thm8.3}'' we take
$0<\eps<\frac{\beta-\alpha}2$ so small that
\begin{equation*}
\forall r>0: W_{\alpha+\eps,\beta-\eps}\cap \J(p)\cap B(0,r)\neq\emptyset.
\end{equation*}
Then we take $r>0$ so small that
\begin{equation}\label{eq:angle}
W_{\alpha+\eps,\beta-\eps}\cap B(0,r)\subset f\left(W_{\alpha,\beta}\right),
\end{equation}
which is possible by conformity of $f$ and $f'(0)=1$. Since the preimages of
$0$ are dense in $\J(p)$, there exists
$\eta\in W_{\alpha+\eps,\beta-\eps}\cap B(0,r)$ and $n\in\N$ such that
$p^{(n)}(\eta)=0$. By \eqref{eq:angle} there exists $\xi\in W_{\alpha,\beta}$
such that $f(\xi)=\eta$, from which we obtain
\begin{equation*}
f(\lambda^n\xi)=p^{(n)}(f(\xi))=p^{(n)}(\eta)=0.
\end{equation*}

For the proof of ``\ref{thm8.2}$\Rightarrow$~\ref{thm8.1}'' we take $z_0\in
W_{\alpha,\beta}$ with $f(z_0)=0$. Then
\begin{equation*}
\forall n\in\N: f(\lambda^{-n}z_0)\in\J(p).
\end{equation*}
For any $r>0$ and $n$ large enough $f(\lambda^{-n}z_0)\in W_{\alpha,\beta}\cap
B(0,r)$, which gives \ref{thm8.1}.
\end{proof}
Similar arguments show
\begin{theorem}\label{thm:zeros-on-line}
Let $p$ be a real polynomial with $p(0)=0$ and $p'(0)=\lambda>1$. Then
\begin{equation}\label{eq:negative-zeros}
  \J(p)\subset\R^-\cup\{0\}\Leftrightarrow \text{ all zeros of }f
\text{ are non-positive real}
\end{equation}
and
\begin{equation}\label{eq:real-zeros}
\J(p)\subset\R\Leftrightarrow \text{ all zeros of }f\text{ are real.}
\end{equation}
\end{theorem}

\section{Real Julia set}\label{sec:real-julia-set}
\begin{lemma}\label{lem1}
  Let $p$ be a real polynomial of degree $d>1$. Then the Julia-set $\J(p)$ is
  real, if and only if there exists an interval $[a,b]$ such that
\begin{equation}\label{eq:inv-int}
p^{(-1)}\left([a,b]\right)\subseteq[a,b].
\end{equation}
\end{lemma}
\begin{proof}
  Assume first that $\J(p)\subset\R$ and take the interval
  $[a,b]=[\min\J(p),\max\J(p)]$. Let $\eps>0$. Since $\J(p)$ is perfect, there
  exist $\xi,\eta\in\J(p)$ with $a<\xi<a+\eps<b-\eps<\eta<b$. All preimages of
  $\xi$ and $\eta$ are in $\J(p)$ by the invariance of $\J(p)$. Furthermore,
  all these preimages are distinct. Therefore, every value $x\in[\xi,\eta]$ has
  exactly $d$ distinct preimages in $[a,b]$ by continuity of $p$. Since $\eps$
  was arbitrary and the two points $a,b$ also have all their preimages in
  $\J(p)\subset[a,b]$, we have proved \eqref{eq:inv-int}.

  Assume on the other hand that $[a,b]$ satisfies \eqref{eq:inv-int}. Since the
  map $p$ has only finitely many critical values, there exists $x\in[a,b]$ such
  that the backward iterates of $x$ are dense in the Julia set. By
  \eqref{eq:inv-int} all these backward iterates are real; therefore $\J(p)$ is
  real.
\end{proof}
\begin{remark}\label{rem1}
By the above proof we can always assume $[a,b]=[\min \J(p),\max
\J(p)]$. Furthermore, we have
\begin{equation*}
 p\left(\{\min \J(p),\max \J(p)\}\right)\subseteq\{\min \J(p),\max \J(p)\},
\end{equation*}
which implies that at least one of the two end points of this interval is
either a fixed point, or they form a cycle of length $2$.
\end{remark}
\begin{theorem}\label{thm1}
  Let $p$ be a polynomial of degree $d>1$ with real Julia set $\J(p)$. Then for
  any fixed point $\xi$ of $p$ with $\min \J(p)<\xi<\max \J(p)$ we have
  $|p'(\xi)|\geq d$. Furthermore, $|p'(\min\J(p))|\geq d^2$ and $|p'(\max
  \J(p))|\geq d^2$. Equality in one of these inequalities implies that $p$ is
  linearly conjugate to the Chebyshev polynomial $T_d$ of degree $d$.
\end{theorem}
\begin{remark}
  This theorem can be compared to
  \cite[Theorem~2]{Buff2003:bieberbach_conjecture_dynamics} and
  \cite{Levin1991:pommerenke's_inequality,
    Pommerenke1986:conformal_mapping_iteration}, where estimates for the
  derivative of $p$ for connected Julia sets are derived. Furthermore, in
  \cite{Eremenko_Levin1992:estimation_characteristic_exponents} estimates for
  $\frac1n\log|(p^{(n)})'(z)|$ for periodic points of period $n$ are given.
\end{remark}
Before we give a proof of the theorem, we present a lemma, which is of some
interest on its own. A similar result is given in 
\cite[Chapter~V, Section~2, Lemma~3]{Levin1980:distribution_zeros_entire}.

\begin{lemma}\label{lem2}
Let $f$ be holomorphic in the angular region $W_{\alpha,\beta}$
If there exists a positive constant $M$ such that
\begin{equation*}
\forall z\in W_{\alpha,\beta}:|f(z)|\geq M,
\end{equation*}
then
\begin{equation*}
\forall \eps>0\,\,\, \exists A,B>0\,\,\, \forall z\in W_{\alpha+\eps,\beta-\eps}:
|f(z)|\leq B\exp(A|z|^\kappa)
\end{equation*}
with $\kappa=\frac\pi{\beta-\alpha}$.
\end{lemma}
\begin{proof}
Without loss of generality we can assume that $M=1$, $\alpha=-\frac\pi2$, and
$\beta=\frac\pi2$. In this case $\kappa=1$. The function
\begin{equation*}
v(z)=\log|f(z)|
\end{equation*}
is a positive harmonic function in the right half-plane. Thus it can be
represented by the Nevanlinna formula
(cf.~\cite[p.100]{Levin1996:lectures_entire_functions})
\begin{equation}\label{eq:nevanlinna}
v(x+iy)=\frac x\pi\int_{-\infty}^\infty\frac{d\nu(t)}{|z-it|^2}+\sigma x,
\end{equation}
where $\nu$ denotes a measure satisfying
\begin{equation*}
\int_{-\infty}^\infty\frac{d\nu(t)}{1+t^2}<\infty
\end{equation*}
and $\sigma\geq0$.

In the region given by $|\arg z|\leq\frac\pi2-\eps$ and $|z|>1$ we have
\begin{equation*}
|z-it|\geq\max(|t|\sin\eps,|z|\sin\eps)\geq\max(1,|t|)\sin\eps.
\end{equation*}
From this it follows that
\begin{equation*}
|z-it|^2\geq\frac12(1+t^2)\sin^2\eps,
\end{equation*}
 which gives
\begin{equation*}
\int_{-\infty}^\infty\frac{d\nu(t)}{|z-it|^2}\leq
\frac2{\sin^2\eps}\int_{-\infty}^\infty\frac{d\nu(t)}{1+t^2}\leq B_\eps
\end{equation*}
for $|z|\geq1$ and some $B_\eps>0$. Setting $A=\frac1\pi B_\eps+\sigma$ and
observing that $x\leq|z|$ completes the proof.
\end{proof}

\begin{proof}[Proof of Theorem~\ref{thm1}]
Without loss of generality we may assume that the fixed point $\xi=0$. Then we
consider the solution $f$ of the Poincar\'e equation
\begin{equation*}
f(\lambda z)=p(f(z))
\end{equation*}
with $\lambda=p'(0)$. We assume first that $\lambda>0$.

First we consider the case $\min \J(p)<\xi<\max \J(p)$. In this case the
function $f(z)/z$ tends to infinity uniformly for $z\to\infty$ in the region
$\eps\leq\arg z\leq\pi-\eps$ for any $\eps>0$ by
Theorem~\ref{thm:poincare-asymp}.  Furthermore, we know that
\begin{equation*}
|f(z)|\geq C\exp(A|z|^{\log_\lambda d})
\end{equation*}
in this region for some positive constants $A$ and $C$. Since
$f(z)/z$ does not vanish at $z=0$, this function satisfies the hypothesis of
Lemma~\ref{lem2}, from which we derive that
\begin{equation*}
\log_\lambda d\leq\frac\pi{\pi-2\eps}
\end{equation*}
holds for any $\eps>0$, which implies $\lambda=p'(0)\geq d$.

The proof in the case $\xi=\max \J(p)$ runs along the same lines. The function
$f(z)/z$ tends to infinity uniformly in any region $|\arg z|\leq\pi-\eps$ in
this case, which by Lemma~\ref{lem2} implies
\begin{equation*}
\log_\lambda d\leq\frac\pi{2\pi-2\eps}
\end{equation*}
for all $\eps>0$, and consequently $\lambda=p'(0)\geq d^2$.

For negative $\lambda=p'(0)$ we apply the same arguments to $p^{(2)}$.

For the proof of the second assertion of the theorem, we first assume that the
fixed point $\xi=0$ satisfies $a=\min \J(p)<0<\max \J(p)=b$ and that $p'(0)=d$.
We know that for a suitable linear conjugate $q$ of the Chebyshev polynomial
$T_d$ we
have $q'(0)=d$ and $\J(q)=[a,b]$ with $0\in(a,b)$.

Let us assume now that $p'(0)=d$ and $\J(p)$ is a Cantor subset of the real
line, or after a rotation that $\J(p)$ is a Cantor subset of the imaginary axis
(this makes notation slightly simpler).

By arguments, similar to those in the beginning of
Section~\ref{sec:furth-analys-peri} we can write
\begin{equation}\label{eq:harmonic}
H(z)=\Re\log g(f(z))=\int_{\J(p)}\log|f(z)-x|\,d\mu(x).
\end{equation}
Since $\Re\log g(.)$ is the Green function of $\J(p)$ with pole at $\infty$
(cf.~\cite[Lemma~9.5.5]{Beardon1991:iteration_rational_functions} or
\cite{Ransford1995:potential_theory_complex_plane}), we know that
$H(z)\geq0$ for all $z\in\C$ and $H(z)=0$, if and only if $f(z)\in\J(p)$ (since
$\K(p)=\J(p)$ in the present case). By Theorem~\ref{thm:poincare-asymp} we have
\begin{equation}\label{eq:periodic}
H(z)=\Re \left(zF(\log_d z)\right)=x\Re(F(\log_d z))-y\Im(F(\log_d z))
\text{ for }z=x+iy,
\end{equation}
and by Theorem~\ref{thm:constant} the function $F$ is not constant in the
present case. The periodic function $\Im F(t+i\phi)$ has zero mean, since the
mean of $F$ is real. Thus $\Im F(t+i\phi)$ attains positive and negative values
for any $\phi$. We now take $z=iy\in i\R^+$ to obtain
\begin{equation*}
H(iy)=-y\Im(F(\log_d y+i\frac\pi{2\log d})).
\end{equation*}
Since $\Im F$ attains positive values by the above argument, we get a
contradiction to $H(z)\geq0$ for all $z$.

A similar argument shows that for $0=\max \J(p)$ and $p'(0)=d^2$ the assumption
that the Julia set is not an interval leads to the same contradiction.
\end{proof}
\begin{remark}\label{rem10}
  Lemma~6.4 in \cite{Derfel_Grabner_Vogl2007:asymptotics_poincare_functions}
  proves Theorem~\ref{thm1} for the special case of quadratic polynomials. The
  proof given in \cite{Derfel_Grabner_Vogl2007:asymptotics_poincare_functions}
  is purely geometrical. 
\end{remark}
\begin{remark}\label{rem11}
We have a purely real analytic proof for
  $|p'(\max\J(p))|\geq d^2$, which is motivated by the proof of the extremality
  of the Chebyshev polynomials of the first kind given in
  \cite{Rivlin1974:chebyshev_polynomials}. However, we could not find a similar
  proof for the other assertions of the theorem.
\end{remark}
\subsection{The Julia set is a subset of the negative reals}
\label{sec:negative-julia-set}
As a consequence of Lemma~\ref{lem2} we get that any solution of the
Poincar\'e equation for a polynomial with Julia set contained in the negative
real axis has order $\leq\frac12$. The only solutions of a Poincar\'e equation
with order $\frac12$ in this situation are the functions
\begin{equation*}
  f(z)=\frac1a\left(\cosh\sqrt{2az}-1\right)
\end{equation*}
for
\begin{equation*}
  p(z)=(T_d(az+1)-1)/a,
\end{equation*}
where $a\in\R^+$ and $T_d$ denotes the Chebyshev polynomial of the first kind
of degree $d$. This is also the only case where the periodic function $F$ in
\eqref{eq:f-asymp} is constant in this situation.

\begin{corollary}\label{cor15}
Assume that $p$ is a real polynomial such that $\J(p)$ is real and 
all coefficients $p_i$ ($i\geq2$) of $p$ are non-negative.
Then $\J(p)\subset\R^-\cup\{0\}$ and therefore
\begin{equation}\label{eq:simple-asymp}
f(z)\sim\exp\left(z^\rho F\left(\frac{\log z}{\log\lambda}\right)\right)
\end{equation}
for $z\to\infty$ and $|\Arg z|<\pi$. Here $F$ is a periodic function of period
$1$ holomorphic in the strip given by $|\Im
w|<\frac\pi{\log\lambda}$. Furthermore, for every $\eps>0$
$\Re e^{i\rho\Arg z}F(\frac{\log z}{\log\lambda})$ is bounded between two
positive constants for $|\Arg z|\leq\pi-\eps$.
\end{corollary}
\begin{proof}
  From
  \cite[Lemmas~6.4~and~6.5]{Derfel_Grabner_Vogl2007:asymptotics_poincare_functions}
  it follows that $f(z)$ has only non-positive real zeros. Then by
  Theorem~\ref{thm:zeros-on-line} $\J(p)\subset\R^-\cup\{0\}$. Finally, the
  assertion follows by applying
  \cite[Theorem~7.5]{Derfel_Grabner_Vogl2007:asymptotics_poincare_functions}.
\end{proof}
\begin{example}
In order to illustrate the above results, we shall turn to the equation
\begin{equation*}
f(5z)=4f(z)^2-3f(z),
\end{equation*}
which arises in the description of Brownian motion on the Sierpi\'nski gasket
\cite{Derfel_Grabner_Vogl2008:zeta_function_laplacian,
  Kroen2002:green_functions_self,
  Kroen_Teufl2004:asymptotics_transition_probabilities,
  Teplyaev2004:spectral_zeta_function}. Here $p(z)=4z^2-3z$, and the fixed
point of interest is $f(0)=1$. This fits into the assumptions of
Section~\ref{sec:assumptions} only after substituting $g(z)=4(f(z)-1)$, where
$g$ satisfies
\begin{equation*}
g(5z)=g(z)^2+5g(z).
\end{equation*}
Now Corollary~\ref{cor15} may be applied to this equation (the preimages of $0$
are real by
\cite[Lemma~6.7]{Derfel_Grabner_Vogl2007:asymptotics_poincare_functions}) to
give \eqref{eq:simple-asymp}.

Note also that $p'(0)=5>4=2^2$ in accordance with Theorem~\ref{thm1}.
\end{example}
\subsection{The Julia set has positive and negative
  elements}\label{sec:julia-set-has}
Again as a consequence of Theorem~\ref{thm1} the solution of the Poincar\'e
equation for a polynomial with real Julia set with positive and negative
elements has order $\leq1$. The only solution of a Poincar\'e equation of order
$1$ in this situation are the functions
\begin{equation*}
f(z)=\frac1a\left(\cos\left(a\frac{z-\frac{2k\pi}{d-1}}
{\sin\frac{k\pi}{d-1}}\right)-\xi_k\right)
\end{equation*}
for
\begin{equation*}
  p(z)=\frac1a\left(T_d(a(z+\xi_k))-\xi_k\right),
\end{equation*}
where $a\in\R^+$ and $\xi_k=\cos\frac{k\pi}{d-1}$ for $1\leq k<\frac{d-1}2$.
This is again the only case where the periodic function $F$ in
\eqref{eq:f-asymp} is constant in this situation.

\section{The Zeta function of the Poincar\'e function}
\label{sec:zeta-funct-poinc}
In \cite{Derfel_Grabner_Vogl2008:zeta_function_laplacian} the zeta function of
a fractal Laplace operator was related to the zeta function of certain
Poincar\'e functions. Asymptotic expansions for the Poincar\'e functions were
then used to give a meromorphic continuation of these zeta functions as well as
information on the location of their poles and values of residues. In this
section we give a generalisation of these results to polynomials whose Fatou
set contains an angular region $W_{-\alpha,\alpha}$ around the positive real
axis. In this case the solution $f$ of \eqref{eq:poincare} has no zeros in an
angular region $W_{-\alpha,\alpha}$. Furthermore, from the Hadamard
factorisation theorem we get
\begin{equation}\label{eq:hadamard}
f(z)=z\exp\left(\sum_{\ell=1}^k(-1)^{\ell-1}\frac{e_\ell z^\ell}\ell\right)
\prod_{\substack{f(-\xi)=0\\\xi\neq0}}\!\!\left(1+\frac z\xi\right)
\exp\left(-\frac z\xi+\frac {z^2}{2\xi^2}+\cdots+
(-1)^{k-1}\frac{z^k}{k\xi^k}\right),
\end{equation}
where $k=\lfloor\log_\lambda d\rfloor$. By the discussion in 
\cite[Section~5]{Derfel_Grabner_Vogl2008:zeta_function_laplacian} the values
$e_1,\ldots,e_k$ are given by the first $k$ terms of the Taylor series of
$\log\frac{f(z)}z$
\begin{equation*}
\log\frac{f(z)}z=\sum_{\ell=1}^k(-1)^{\ell-1}\frac{e_\ell z^\ell}\ell
+\BigOh(z^{k+1}).
\end{equation*}

The zeta function of $f$ is now defined as
\begin{equation}\label{eq:zeta_f}
\zeta_f(s)=\sum_{\substack{f(-\xi)=0\\\xi\neq0}}\xi^{-s},
\end{equation}
where $\xi^{-s}$ is defined using the principal value of the logarithm, which
is sensible, since $\xi$ is never negative real by our assumption on
$\F_\infty(p)$. The function $\zeta_f(s)$ is holomorphic in the half plane $\Re
s>\rho$. In \cite{Derfel_Grabner_Vogl2008:zeta_function_laplacian} we used the
equation
\begin{equation}\label{eq:zeta-weierstrass}
\int_0^\infty\left(\log f(x)-\log x-
\sum_{\ell=1}^k(-1)^{\ell-1}\frac{e_\ell x^\ell}\ell\right)
x^{-s-1}\,dx=\zeta_f(s)\frac\pi{s\sin\pi s},
\end{equation}
which holds for $\rho<\Re s<k+1$, to derive the existence of a
meromorphic continuation of $\zeta_f$ to the whole complex plane. There
(\cite[Theorem~8]{Derfel_Grabner_Vogl2008:zeta_function_laplacian}) we obtained
\begin{equation*}
\Res_{s=\rho+\frac{2k\pi i}{\log\lambda}}\zeta_f(s)=
-\frac{f_k}\pi\left(\rho+\frac{2\pi ik}{\log\lambda}\right)
\sin\pi\left(\rho+\frac{2\pi ik}{\log\lambda}\right),
\end{equation*}
where $f_k$ is given by \eqref{eq:fk}. From this we get
\begin{equation}\label{eq:Res-zeta_f}
\Res_{s=\rho+\frac{2k\pi i}{\log\lambda}}\zeta_f(s)=
-\Res_{s=-\rho-\frac{2k\pi i}{\log\lambda}}M_\mu(s).
\end{equation}
This shows that the function
\begin{equation}\label{eq:zeta_f-M_mu}
\zeta_f(s)-M_\mu(-s)
\end{equation}
is holomorphic in $\rho-1<\Re s<\rho+1$, since the single poles on the line
$\Re s=\rho$ cancel. This fact was used in
\cite{Grabner1997:functional_iterations_stopping} to derive an analytic
continuation for $\zeta_f(s)$.

\begin{theorem}\label{thm9}
  Let $f$ be the entire solution of \eqref{eq:poincare} and assume that $p$ is
  neither linearly conjugate to a Chebyshev polynomial nor to a monomial and
  that $W_{-\alpha,\alpha}\subset\F_\infty(p)$ for some $\alpha>0$. Then the
  following assertions hold
  \begin{enumerate}
  \item\label{enum2} the limit $\lim_{t\to\infty}t^{-\rho}\log f(t)$ does not
    exist.
  \item\label{enum1} $\zeta_f(s)$ has at least two non-real poles in the
    set $\rho+2\pi i\sigma\Z$
    \textup{(}$\sigma=\frac1{\log\lambda}$\textup{)}.
  \item\label{enum5} the limit $\lim_{x\to0}x^{-\rho}G(x)$ with $G$ given by
    \eqref{eq:G(w)} does not exist.
  \end{enumerate}
\end{theorem}
\begin{proof}
  Equation \eqref{eq:f-asymp} in Theorem \ref{thm:poincare-asymp} (see also
  \cite{Derfel_Grabner_Vogl2007:asymptotics_poincare_functions}) implies that
  \begin{equation*}
  z^{-\rho}\log f(z)=F(\log_\lambda z)+o(1)\text{ for }z\to\infty\text{ and }
  z\in W_{-\alpha,\alpha}
\end{equation*}
with a periodic function $F$ of period $1$.
  Theorem~\ref{thm:constant} implies that $F$ is a non-constant . Thus the
  limit in \ref{enum2} does not exist.

Since the periodic function $F$ is non-constant, there exists a $k_0\neq0$ such
that the Fourier-coefficients $f_{\pm k_0}$ do not vanish. By
\eqref{eq:f-asymp} we have
\begin{equation*}
\log f(z)=z^\rho\sum_{k\in\Z}f_k z^{\frac{2k\pi i}{\log\lambda}}+
\BigOh(z^{-M})
\end{equation*}
for any $M>0$. By properties of the Mellin transform
(cf.~\cite{Paris_Kaminski2001:asymptotics_mellin_barnes}), every term
$Az^{\rho+i\tau}$ in the asymptotic expansion of $\log f(z)$ corresponds to a
first order pole of the Mellin transform of $\log f(z)$ with residue $A$ at
$s=\rho+i\tau$. Since $f_{k_0}\neq0$, from \eqref{eq:zeta-weierstrass} we have
simple poles of $\zeta_f(s)$ at $s=\rho\pm\frac{2k_0\pi i}{\log\lambda}$.

Assertion \ref{enum5} follows from \ref{enum2} by \eqref{eq:logg}.
\end{proof}

In the following we consider the zero counting function of $f$
\begin{equation}\label{eq:N(x)}
N_f(x)=\sum_{\substack{|\xi|<x\\f(\xi)=0}}1.
\end{equation}

\begin{theorem}\label{thm11}
  Let $f$ be the entire solution of \eqref{eq:poincare}. Then the following are
  equivalent 
  \begin{enumerate}
  \item\label{enum3} the limit $\lim_{x\to\infty}x^{-\rho}N_f(x)$ does
    not exist.
  \item\label{enum4} the limit $\lim_{t\to0}t^{-\rho}\mu(B(0,t))$ does not
    exist.
  \end{enumerate}
\end{theorem}
\begin{proof}
  For the proof of the equivalence of \ref{enum3} and \ref{enum4} we
  observe that by the fact that $f'(0)=1$, there is an $r_0>0$ such that
  $f:B(0,r_0)\to\C$ is invertible.  For the following we choose
  $n=\lfloor\log_\lambda(x/r_0)\rfloor+k$ and let the integer $k>0$ be fixed
  for the moment. Then we use the functional equation for $f$ to get
\begin{equation*}
N_f(x)=\#\left\{\xi\mid f(\lambda^n\xi)=p^{(n)}(f(\xi))=0
\wedge|\xi|<x\lambda^{-n}\right\}
\!=\!\#\!\left(p^{(-n)}(0)\cap f(B(0,x\lambda^{-n}))\right).
\end{equation*}
This last expression can now be written in terms of the discrete measure
$\mu_n$ given in \eqref{eq:mu_n}
\begin{equation*}
N_f(x)=d^n\mu_n\left(f(B(0,x\lambda^{-n}))\right).
\end{equation*}
By the weak convergence of the measures $\mu_n$
(cf.~\cite{Brolin1965:invariant_sets_under}) we get for $x\to\infty$
(equivalently $n\to\infty$)
\begin{equation*}
N_f(x)=d^n\mu(f(B(0,x\lambda^{-n})))+o(d^n)=
x^\rho(x\lambda^{-n})^{-\rho}\mu(f(B(0,x\lambda^{-n})))+o(x^\rho).
\end{equation*}
By our choice of $n$ we have $r_0\lambda^{-k-1}\leq x\lambda^{-n}\leq
r_0\lambda^{-k-1}$, which makes the first term dominant. From this it is clear
that the existence of the limit
\begin{equation*}
\lim_{x\to\infty}x^{-\rho}N_f(x)=C
\end{equation*}
is equivalent to
\begin{equation*}
\mu(f(B(0,t)))=Ct^\rho\text{ for }r_0\lambda^{-k}\leq t< r_0\lambda^{-(k-1)}.
\end{equation*}
Since $k$ was arbitrary this implies
\begin{equation}\label{eq:mu(B)}
\mu(f(B(0,t)))=Ct^\rho\text{ for }0< t< r_0.
\end{equation}

It follows from $f'(0)=1$ that
\begin{equation}\label{eq:mu(f(B))-mu(B))}
\forall\eps>0:\exists\delta>0:\forall t<\delta:
B(0,(1-\eps)t)\subset f(B(0,t))\subset B(0,(1+\eps)t).
\end{equation}
Thus the existence of the limit in assertion \ref{enum4} is equivalent to
\begin{equation*}
\lim_{t\to0}t^{-\rho}\mu(f(B(0,t)))=C.
\end{equation*}
Thus \ref{enum3} and \ref{enum4} are equivalent.
\end{proof}

\begin{remark}
If $\J(p)$ is real and disconnected then the limits in Theorem~\ref{thm11} do
not exist. Furthermore, it is known that the limit
\begin{equation*}
\lim_{t\to0}t^{-\rho}\mu(f(B(w,t)))=C
\end{equation*}
does not exist for $\mu$-almost all $w\in\J(p)$
(cf.~\cite[Theorem~14.10]{Mattila1995:geometry_sets_measures}), if $\rho$ is
not an integer.
\end{remark}
This motivates the following conjecture.
\begin{conj*}
The limits in Theorem~\ref{thm11} exist, if and only if $p$ is either linearly
conjugate to a Chebyshev polynomial or a monomial.
\end{conj*}

\begin{acknowledgements}This research was initiated during the second author's
  visit to the Ben Gurion University of the Negev with support by the Center of
  Advanced Studies in Mathematics.\\
  It was completed during the second author's visit at the Center for
  Constructive Approximation and the Department of Mathematics at Vanderbilt
  University, Nashville, Tennessee. He is especially thankful to
  Edward~B.~Saff for the invitation and the great hospitality.\\
  The first author wants to thank Alexandre~Er\"emenko, Genadi~Levin, and
  Mikhail~Sodin for interesting discussions.\\
  The authors are indebted to an anonymous referee for valuable remarks.
\end{acknowledgements}


\begin{thebibliography}{10}

\bibitem{Barlow1998:diffusions_on_fractals}
{\bibname M.~T. Barlow}.
\newblock Diffusions on fractals.
\newblock In \emph{Lectures on probability theory and statistics (Saint-Flour,
  1995)}, pages 1--121 (Springer Verlag, Berlin, 1998).

\bibitem{Barlow_Perkins1988:brownian_motion_sierpinski}
{\bibname M.~T. Barlow \and  E.~A. Perkins}.
\newblock Brownian motion on the {S}ierpi\'nski gasket.
\newblock \emph{Probab. Theory Relat. Fields} \textbf{79} (1988) 543--623.

\bibitem{Beardon1991:iteration_rational_functions}
{\bibname A.~Beardon}.
\newblock \emph{Iteration of {R}ational {F}unctions} (Springer, Berlin, New
  York, 1991).

\bibitem{Bessis_Geronimo_Moussa1984:mellin_transforms_associated}
{\bibname D.~Bessis, J.~S. Geronimo \and  P.~Moussa}.
\newblock Mellin transforms associated with {J}ulia sets and physical
  applications.
\newblock \emph{J. Statist. Phys.} \textbf{34} (1984) 75--110.

\bibitem{Blanchard1984:complex_analytic_dynamics}
{\bibname P.~Blanchard}.
\newblock Complex analytic dynamics on the {R}iemann sphere.
\newblock \emph{Bull. Amer. Math. Soc. (N.S.)} \textbf{11} (1984) 85--141.

\bibitem{Boettcher1905:beitraege_zur_theorie}
{\bibname L.~B\"ottcher}.
\newblock Beitr\"age zur {T}heorie der {I}terationsrechnung.
\newblock \emph{Bull. Kasan Math. Soc.} \textbf{14} (1905) 176.

\bibitem{Brolin1965:invariant_sets_under}
{\bibname H.~Brolin}.
\newblock Invariant sets under iteration of rational functions.
\newblock \emph{Ark. Mat.} \textbf{6} (1965) 103--144.

\bibitem{Buff2003:bieberbach_conjecture_dynamics}
{\bibname X.~Buff}.
\newblock On the {B}ieberbach conjecture and holomorphic dynamics.
\newblock \emph{Proc. Amer. Math. Soc.} \textbf{131} (2003) 755--759
  (electronic).

\bibitem{Derfel_Grabner_Vogl2007:asymptotics_poincare_functions}
{\bibname G.~Derfel, P.~J. Grabner \and  F.~Vogl}.
\newblock Asymptotics of the {P}oincar\'e functions.
\newblock In {\bibname D.~Dawson, V.~Jaksic \and  B.~Vainberg}, editors,
  \emph{Probability and Mathematical Physics: A Volume in Honor of Stanislav
  Molchanov}, volume~42 of \emph{CRM Proceedings and Lecture Notes}, pages
  113--130 (Centre de Recherches Math\'ematiques, Montreal, 2007).

\bibitem{Derfel_Grabner_Vogl2008:zeta_function_laplacian}
{\bibname G.~Derfel, P.~J. Grabner \and  F.~Vogl}.
\newblock The {Z}eta function of the {L}aplacian on certain fractals.
\newblock \emph{Trans. Amer. Math. Soc.} \textbf{360} (2008) 881--897
(electronic).

\bibitem{Doetsch1971:handbuch_der_laplace}
{\bibname G.~Doetsch}.
\newblock \emph{Handbuch der {L}aplace-{T}ransformation. {B}and {I}: {T}heorie
  der {L}aplace-{T}ransformation} (Birkh\"auser Verlag, Basel, 1971).
\newblock Verbesserter Nachdruck der ersten Auflage 1950, Lehrb\"ucher und
  Monographien aus dem Gebiete der exakten Wissenschaften. Mathematische Reihe,
  Band 14.

\bibitem{Eremenko_Levin1989:periodic_points_polynomials}
{\bibname A.~{\`E}. Er{\"e}menko \and  G.~M. Levin}.
\newblock Periodic points of polynomials ({R}ussian).
\newblock \emph{Ukrain. Mat. Zh.} \textbf{41} (1989) 1467--1471, 1581.
\newblock Translation in Ukrainian Math. J. \textbf{41} (1989), 1258--1262.

\bibitem{Eremenko_Levin1992:estimation_characteristic_exponents}
{\bibname A.~{\`E}. Er{\"e}menko \and  G.~M. Levin}.
\newblock Estimation of the characteristic exponents of a polynomial
  ({R}ussian).
\newblock \emph{Teor. Funktsi\u\i{} Funktsional. Anal. i Prilozhen.} pages
  30--40 (1993).
\newblock Translation in J. Math. Sci. (New York) \textbf{85} (1997),
  2164--2171.

\bibitem{Eremenko_Sodin1990:iterations_rational_functions}
{\bibname A.~{\`E}. Er{\"e}menko \and  M.~L. Sodin}.
\newblock Iterations of rational functions and the distribution of the values
  of {P}oincar\'e functions ({R}ussian.
\newblock \emph{Teor. Funktsi\u\i{} Funktsional. Anal. i Prilozhen.}
  \textbf{53} (1990) 18--25.
\newblock Translation in J. Soviet Math. \textbf{58} (1992), 504--509.

\bibitem{Falconer2003:fractal_geometry}
{\bibname K.~J. Falconer}.
\newblock \emph{Fractal {G}eometry} (John Wiley \& Sons Inc., Hoboken, NJ,
  2003).
\newblock Mathematical foundations and applications.

\bibitem{Garnett_Marshall2005:harmonic_measure}
{\bibname J.~B. Garnett \and  D.~E. Marshall}.
\newblock \emph{Harmonic {M}easure}, volume~2 of \emph{New Mathematical
  Monographs} (Cambridge University Press, Cambridge, 2005).

\bibitem{Goluzin1969:geometric_theory_functions}
{\bibname G.~M. Goluzin}.
\newblock \emph{Geometric {T}heory of {F}unctions of a {C}omplex {V}ariable}.
\newblock Translations of Mathematical Monographs, Vol. 26 (American
  Mathematical Society, Providence, R.I., 1969).

\bibitem{Grabner1997:functional_iterations_stopping}
{\bibname P.~J. Grabner}.
\newblock Functional iterations and stopping times for {B}rownian motion on the
  {S}ierpi\'nski gasket.
\newblock \emph{Mathematika} \textbf{44} (1997) 374--400.

\bibitem{Hamilton1995:length_julia_curves}
{\bibname D.~H. Hamilton}.
\newblock Length of {J}ulia curves.
\newblock \emph{Pacific J. Math.} \textbf{169} (1995) 75--93.

\bibitem{Harris1963:theory_branching_processes}
{\bibname T.~E. Harris}.
\newblock \emph{The {T}heory of {B}ranching {P}rocesses} (Springer, Berlin, New
  York, 1963).

\bibitem{Ishizaki_Yanagihara2005:borel_and_julia}
{\bibname K.~Ishizaki \and  N.~Yanagihara}.
\newblock Borel and {J}ulia directions of meromorphic {S}chr\"oder functions.
\newblock \emph{Math. Proc. Cambridge Philos. Soc.} \textbf{139} (2005)
  139--147.

\bibitem{Koenigs1884:recherches_sur_integrales}
{\bibname G.~Koenigs}.
\newblock Recherches sur les int\'egrales de certaines \'equations
  fonctionelles.
\newblock \emph{Ann. Sci. Ec. Norm. Super. III. Ser.} \textbf{1} (1884) 3--41.

\bibitem{Koenigs1885:nouvelles_recherches_sur}
{\bibname G.~Koenigs}.
\newblock Nouvelles recherches sur les Žequations fonctionnelles.
\newblock \emph{Ann. Sci. Ec. Norm. Super. III. Ser.} \textbf{2} (1885)
  385--404.

\bibitem{Kroen2002:green_functions_self}
{\bibname B.~Kr{\"o}n}.
\newblock Green functions on self-similar graphs and bounds for the spectrum of
  the {L}aplacian.
\newblock \emph{Ann. Inst. Fourier (Grenoble)} \textbf{52} (2002) 1875--1900.

\bibitem{Kroen_Teufl2004:asymptotics_transition_probabilities}
{\bibname B.~Kr{\"o}n \and  E.~Teufl}.
\newblock Asymptotics of the transition probabilities of the simple random walk
  on self-similar graphs.
\newblock \emph{Trans. Amer. Math. Soc.} \textbf{356} (2004) 393--414
  (electronic).

\bibitem{Kuczma_Choczewski_Ger1990:iterative_functional_equations}
{\bibname M.~Kuczma, B.~Choczewski \and  R.~Ger}.
\newblock \emph{Iterative {F}unctional {E}quations}, volume~32 of
  \emph{Encyclopedia of Mathematics and its Applications} (Cambridge University
  Press, Cambridge, 1990).

\bibitem{Levin1980:distribution_zeros_entire}
{\bibname B.~Y. Levin}.
\newblock \emph{Distribution of {Z}eros of {E}ntire {F}unctions}, volume~5 of
  \emph{Translations of Mathematical Monographs} (American Mathematical
  Society, Providence, R.I., 1980), revised edition.
\newblock Translated from the Russian by R. P. Boas, J. M. Danskin, F. M.
  Goodspeed, J. Korevaar, A. L. Shields and H. P. Thielman.

\bibitem{Levin1996:lectures_entire_functions}
{\bibname B.~Y. Levin}.
\newblock \emph{Lectures on {E}ntire {F}unctions}, volume 150 of
  \emph{Translations of Mathematical Monographs} (American Mathematical
  Society, Providence, RI, 1996).
\newblock In collaboration with and with a preface by Yu.\ Lyubarskii, M. Sodin
  and V. Tkachenko, Translated from the Russian manuscript by V. Tkachenko.

\bibitem{Levin1991:pommerenke's_inequality}
{\bibname G.~M. Levin}.
\newblock On {P}ommerenke's inequality for the eigenvalues of fixed points.
\newblock \emph{Colloq. Math.} \textbf{62} (1991) 167--177.

\bibitem{Lindstroem1990:brownian_motion_nested}
{\bibname T.~Lindstr{\o}m}.
\newblock \emph{Brownian {M}otion on {N}ested {F}ractals}, volume 420 of
  \emph{Mem. Amer. Math. Soc.} (Amer. Math. Soc., 1990).

\bibitem{Malozemov_Teplyaev2003:self_similarity_operators}
{\bibname L.~Malozemov \and  A.~Teplyaev}.
\newblock Self-similarity, operators and dynamics.
\newblock \emph{Math. Phys. Anal. Geom.} \textbf{6} (2003) 201--218.

\bibitem{Mattila1995:geometry_sets_measures}
{\bibname P.~Mattila}.
\newblock \emph{Geometry of {S}ets and {M}easures in {E}uclidean {S}paces,
  {F}ractals and {R}ectifiability}, volume~44 of \emph{Cambridge Studies in
  Advanced Mathematics} (Cambridge University Press, Cambridge, 1995).

\bibitem{Milnor2006:dynamics_complex}
{\bibname J.~Milnor}.
\newblock \emph{Dynamics in one complex variable}, volume 160 of \emph{Annals
  of Mathematics Studies} (Princeton University Press, Princeton, NJ, 2006),
  third edition.

\bibitem{Oberhettinger1974:tables_mellin_transforms}
{\bibname F.~Oberhettinger}.
\newblock \emph{Tables of {M}ellin {T}ransforms} (Springer-Verlag, New York,
  1974).

\bibitem{Paris_Kaminski2001:asymptotics_mellin_barnes}
{\bibname R.~B. Paris \and  D.~Kaminski}.
\newblock \emph{Asymptotics and {M}ellin-{B}arnes {I}ntegrals}, volume~85 of
  \emph{Encyclopedia of Mathematics and its Applications} (Cambridge University
  Press, Cambridge, 2001).

\bibitem{Poincare1886:une_classe_etendue}
{\bibname H.~Poincar\'e}.
\newblock Sur une classe \'etendue de transcendantes uniformes.
\newblock \emph{C. R. Acad. Sci. Paris} \textbf{103} (1886) 862--864.

\bibitem{Poincare1890:une_classe_nouvelle}
{\bibname H.~Poincar\'e}.
\newblock Sur une classe nouvelle de transcendantes uniformes.
\newblock \emph{J. Math. Pures Appl. IV. Ser.} \textbf{6} (1890) 316--365.

\bibitem{Pommerenke1986:conformal_mapping_iteration}
{\bibname C.~Pommerenke}.
\newblock On conformal mapping and iteration of rational functions.
\newblock \emph{Complex Variables Theory Appl.} \textbf{5} (1986) 117--126.

\bibitem{Ransford1995:potential_theory_complex_plane}
{\bibname T.~Ransford}.
\newblock \emph{Potential {T}heory in the {C}omplex {P}lane}, volume~28 of
  \emph{London Mathematical Society Student Texts} (Cambridge University Press,
  Cambridge, 1995).

\bibitem{Rivlin1974:chebyshev_polynomials}
{\bibname T.~J. Rivlin}.
\newblock \emph{The {C}hebyshev {P}olynomials} (Wiley-Interscience [John Wiley
  \& Sons], New York, 1974).
\newblock Pure and Applied Mathematics.

\bibitem{Romanenko_Sharkovsky2000:long_time_properties}
{\bibname E.~Romanenko \and  A.~Sharkovsky}.
\newblock Long time properties of solutions of simplest $q$-difference
  equations ({R}ussian) (2000).
\newblock Preprint.

\bibitem{Schroeder1871:uber_iterierte_funktionen}
{\bibname E.~Schr\"oder}.
\newblock \"{U}ber iterierte {F}unktionen.
\newblock \emph{Math. Ann.} \textbf{3} (1871) 296--322.

\bibitem{Teplyaev2004:spectral_zeta_function}
{\bibname A.~Teplyaev}.
\newblock Spectral zeta function of symmetric fractals.
\newblock In {\bibname C.~Bandt, U.~Mosco \and  M.~Z{\"a}hle}, editors,
  \emph{Fractal geometry and stochastics III}, volume~57 of \emph{Progr.
  Probab.}, pages 245--262 (Birkh\"auser, Basel, 2004).

\bibitem{Teplyaev2007:spectral_zeta_functions}
{\bibname A.~Teplyaev}.
\newblock Spectral zeta functions of fractals and the complex dynamics of
  polynomials.
\newblock \emph{Trans. Amer. Math. Soc.} \textbf{359} (2007) 4339--4358
  (electronic).

\bibitem{Valiron1923:lectures_on_general}
{\bibname G.~Valiron}.
\newblock \emph{Lectures on the {G}eneral {T}heory of {I}ntegral {F}unctions}
  (Private, Toulouse, 1923).

\bibitem{Valiron1954:fonctions_analytiques}
{\bibname G.~Valiron}.
\newblock \emph{Fonctions {A}nalytiques} (Presses Universitaires de France,
  Paris, 1954).

\end{thebibliography}
\end{document}